\documentclass[11pt,reqno]{amsart}
\usepackage{tikz}
\usepackage{subfig}
\usepackage{epstopdf}
\usepackage{epsfig}

\usepackage{tikz}
\usetikzlibrary{shapes,arrows}
\newcommand{\dd}{\mathcal{\dagger}}
\newcommand{\ts}{\mathsf{T}}

\tikzstyle{decision} = [diamond, draw, fill=blue!20, 
    text width=4.5em, text badly centered, node distance=3cm, inner sep=0pt]
\tikzstyle{block} = [rectangle, draw, fill=blue!20, 
    text width=5em, text centered, rounded corners, minimum height=4em]
\tikzstyle{line} = [draw, -latex']
\tikzstyle{cloud} = [draw, ellipse,fill=red!20, node distance=3cm,
    minimum height=2em]

 \newtheoremstyle{mystyle}{36pt}{}{}{}{\bfseries}{.}{ }{}
 
  \theoremstyle{plain}
 \newtheorem{theorem}{Theorem}
 
 \newtheorem{definition}[]{Definition}
 
 \newtheorem{proposition}[]{Proposition}
 \newtheorem{example}{Example}

  \theoremstyle{remark}
 \newtheorem{remark}{Remark}

\usetikzlibrary{positioning}
\tikzset{main node/.style={circle,fill=blue!20,draw,minimum size=1cm,inner sep=0pt},  }

\begin{document}
\title{Wasserstein Hamiltonian flows}
\author[Chow, Li, Zhou]{Shui-Nee Chow, Wuchen Li and Haomin Zhou}
\keywords{Optimal transport; Density manifold; Hamiltonian flow.}
\maketitle
\begin{abstract}
We establish kinetic Hamiltonian flows in density space embedded with the $L^2$-Wasserstein metric tensor. We derive the Euler-Lagrange equation in density space, which introduces the associated Hamiltonian flows. We demonstrate that many classical equations, such as Vlasov equation, Schr{\"o}dinger equation and Schr{\"o}dinger bridge problem, can be rewritten as the formalism of Hamiltonian flows in density space.  
\end{abstract}
\section{Introduction}
In recent years, optimal transport theory provides essential tools for partial differential equations \cite{vil2003, vil2008}. It introduces a type of distance functions in the space of probability densities, which evolve differential structures in the underlying sample space. A particular distance function, named $L^2$-Wasserstein distance, exhibits the metric tensor structure. The density space with this metric forms an infinite-dimensional Riemannian manifold, named density manifold \cite{Lafferty}. Many well-known density equations are gradient flows in density manifold \cite{Otto}. A famous example is a Fokker-Planck equation with gradient drift vector field. It mathematically demonstrates an intuition: The density of gradient flow in sample space is gradient flow in density manifold. 

Despite various successful studies of gradient flows, the other essential flows in density manifold, Hamiltonian flows, are not completely clear. See a detailed discussion in page 253 of \cite{vil2003}. It is because that a typical kinetic Hamiltonian flow in manifold (including density manifold) relies on the associated Christoffel symbol \cite{li_geometry}. This paper takes a natural first step in this direction. Following key ideas in \cite{li_SE, li_geometry, Nelson0}, we establish the formalism of Hamiltonian flows for density manifold in Theorem \ref{ELD}. It directly follows from the variational principle in tangent bundles of density manifold. In other words, we propose to study the following second order equation:
\begin{equation*}
\partial_{tt}\rho_t-\Big(\Delta_{\partial_t\rho_t}\Delta_{\rho_t}^{\dagger}\partial_t\rho_t+\frac{1}{2}\Delta_{\rho_t}(\nabla \Delta_{\rho_t}^{\dagger}\partial_t\rho_t)^2\Big)=\nabla\cdot(\rho_t\nabla\frac{\delta}{\delta\rho_t}\mathcal{F}(\rho_t)).
\end{equation*}
where $\rho(t,x):=\rho_t$ is the density function, $\partial_{tt}$ is the second time derivative, $\Delta_{\rho}=\nabla\cdot(\rho\nabla )$ is an elliptic operator, and $\mathcal{F}(\rho)$ is a given energy functional. Here the coefficient of quadratic formulation for $\partial\rho_t$ is the Christoffel symbol in density manifold. Given various energies, we will show that the above equation is the other formulation of many classical equations, including Vlasov equation, Schr{\"o}dinger equation and Schr{\"o}dinger bridge problem.  

In literature, the study of Hamiltonian flows in density manifold follows Nelson's stochastic mechanics \cite{Carlen, Nelson0, Nelson1, Nelson, Nelson3}. See related work in \cite{CP}. Along with this framework,  Lafferty introduces the Riemannian manifold structure of density space. See \cite{Lafferty} or section 3 of \cite{Nelson3}. Nowadays this metric tensor is named $L^2$-Wasserstein metric, known in optimal transport communities \cite{Otto, vil2003, vil2008}. In classical approaches, the Hamiltonian flow in density space is induced by the vector field in sample space. It relies on the cotangent bundle (dual coordinates) of density manifold, which is often named Otto calculus \cite{vil2008}.   
In contrast to their work, our approach considers the other direction. We use the vector field in density space to describe the one in sample space. This approach applies the tangent bundle of density manifold \cite{li_geometry}. From this angle, we introduce the Lagrangian formalism of density manifold. 

The plan of paper is as follows. In section \ref{sec2}, we review the formulation of Hamiltonian flows with associated Christoffel symbol on Riemannian manifolds. In section \ref{sec3}, we derive the ones in density manifold. Several examples are demonstrated in section \ref{sec4}. 
\section{Hamiltonian flows on Riemannian manifolds}\label{sec2}
In this section, we briefly review classical Hamiltonian flows on a finite dimensional Riemannian manifold. It provides us the intuition to derive the ones in density space. 

Let $(M, g)$ be a smooth, compact, $d$-dimensional Riemannian manifold without boundaries. Here $g$ is the metric tensor of $M$. Given a smooth potential function $F\colon M\rightarrow \mathbb{R}$, a classical Hamiltonian flow in $(M,g)$ refers to the following second order differential equation
\begin{equation}\label{HF}
\ddot x+\Gamma(\dot x, \dot x)=-\textrm{grad}f(x),
\end{equation}
where $x=(x_i)_{i=1}^d$ is a local coordinate in $M$, $\dot x=\frac{dx(t)}{dt}$, $\Gamma(\dot x, \dot x)=\Big(\sum_{1\leq i,j\leq d}\Gamma_{ij}^k(x)\dot{x}_i\dot{x}_j\Big)_{k=1}^d$, $\Gamma_{ij}^k(x)$ is the Christoffel symbol, which is the coefficient of the quadratic term of $\dot x$, $f\colon M\rightarrow \mathbb{R}$ is a given potential function, and $\textrm{grad}$ is the Riemannian gradient operator. 

We next illustrate equation \eqref{HF} by using Hamilton's variational principle. We will explain what is the Lagrangian formalism of equation \eqref{HF}, what is its explicit formulation and why does it relate to Hamilton's equations. The Lagrangian is the function $L$ defined by 
 $$L(x,\dot x)=\frac{1}{2}{\dot x}^{\ts}g(x)\dot x-f(x).$$
In above, $L$ represents the kinetic energy minus the potential energy $f$. Here the metric tensor $g(x)\in \mathbb{R}^{d\times d}$ is introduced in kinetic energy. Consider a variational problem in $(M,g)$ by 
\begin{equation*}
I(x(t))=\inf_{x(t)}\big\{\int_0^TL(x,\dot x)dt\colon x(0)=x_0,~x(T)=x_T\big\}.
\end{equation*}
A path is critical for $L$ in case $I(x(t))$ is stationary for variations. It satisfies the Euler-Lagrange equation
\begin{equation}\label{ELF}
\frac{d}{dt}\frac{d}{d\dot x}L(x,\dot x)=\frac{d}{dx}L(x,\dot x).
\end{equation}

In fact, the trajectory of Hamiltonian flow is a critical path. In other words, equation \eqref{HF} can be derived by expressing \eqref{ELF} explicitly. Substituting $\frac{d}{d\dot x}L(x,\dot x)=g(x)\dot x$ into \eqref{ELF}, 
\begin{equation*}
\frac{d}{dt}(g(x)\dot x)=g(x)\ddot x+(d_xg_{ij}(x)\dot x)_{1\leq i,j\leq d}\dot x=\frac{1}{2} \dot x^{T}d_xg(x)\dot x-d_xf(x).
\end{equation*} 
By multiplying $g(x)^{-1}$ on both sides and collecting all the quadratic terms of $\dot x$ in above equation, then
\begin{equation*}
\ddot x+ g(x)^{-1} \Big((d_{x_k}g_{ij}(x)\dot x)_{1\leq i,j\leq d}\dot x-\frac{1}{2}\dot x^{T}d_{x_k}g(x)\dot x\Big)_{k=1}^n=-g(x)^{-1}d_xf(x).
\end{equation*}
Comparing the above equation with \eqref{HF}, the explicit formulation of geometric formulas are derived:
\begin{equation*}
\Gamma(\dot x, \dot x)=g(x)^{-1} \Big((d_{x_k}g_{ij}(x)\dot x)_{1\leq i,j\leq d}\dot x-\frac{1}{2}\dot x^{T}d_{x_k}g(x)\dot x\Big)_{k=1}^n,
\end{equation*}
 and 
 \begin{equation*}
  \textrm{grad}f=g(x)^{-1}d_xf(x),
 \end{equation*}
 where $d_x$ is the differential operator. 

Moreover, there is a Hamiltonian structure for each critical path. In other words, equation \eqref{HF} forms a first order ODE system, which is with the Hamiltonian vector field 
(a symplectic matrix times the differential of Hamiltonian). Consider the Legendre transformation 
\begin{equation*}
p=g(x)\dot x.
\end{equation*}
Here, $(x, \dot x)$ refers to the primal coordinates in the tangent bundle while $(x, p)$ represents the dual coordinates in the cotangent bundle.
The flow in primal coordinates can be recast as the first order ODE in dual coordinates. In other words,
\begin{equation*}
\begin{pmatrix}
\dot x\\
\dot p
\end{pmatrix}=\begin{pmatrix}
0 & \mathbb{I}\\
-\mathbb{I} & 0
\end{pmatrix}d_{x,p}H(x, p),
\end{equation*}
where $\begin{pmatrix}
0 & \mathbb{I}\\
-\mathbb{I} & 0
\end{pmatrix}$ is named the symplectic matrix, $d$ is the differential operator
and $H$ is the Hamiltonian function
 $$H(x,p)=\frac{1}{2}p^{\ts}g(x)^{-1}p+f(x)=\frac{1}{2}\dot x^{\ts}g(x)\dot x+f(x).$$ 
We note that $H$ is the summation of kinetic energy and potential energy. Based on above known facts, we introduce Hamiltonian flows in density manifold. 
\section{Hamiltonian flows on density manifold}\label{sec3}
In this section, we derive the Hamiltonian flow in density space with respect to the $L^2$-Wasserstein metric tensor.  

\subsection{$L^2$-Wasserstein metric tensor}
We first review some facts. Consider the space of positive smooth density functions supported on $M$.
\begin{equation*}
\mathcal{P}_+(M)=\{\rho d\textrm{vol}_M\colon \rho\in C^{\infty}(M),~\rho>0,~\int_M\rho d\textrm{vol}_M=1\}.
\end{equation*}
Denote the tangent space at $\rho \in \mathcal{P}_+(M)$ by 
\begin{equation*}
T_\rho \mathcal{P}_+({M})=\{\sigma\in C^{\infty}(M)\colon\int_{M}\sigma d\textrm{vol}_M=0\}.
\end{equation*}

The $L^2$-Wasserstein metric tensor is defined as follows. Denote the space of potential functions on $M$ by $\mathcal{F}(M)$. Consider the quotient space 
\begin{equation*}
\mathcal{F}(M)/ \mathbb{R}=\{[\Phi]\mid \Phi\in C^{\infty}(M)\},
\end{equation*}
where $[\Phi]=\{\Phi+c\mid c\in\mathbb{R}\}$ are functions defined up to addition of constants. 

The identification map is defined by
\begin{equation*}
\mathbf{V}\colon\mathcal{F}(M)/\mathbb{R} \rightarrow T_\rho\mathcal{P}_+(M),\quad\quad 
\mathbf{V}_\Phi=-\nabla\cdot(\rho\nabla\Phi).
\end{equation*} 
Since $M$ is a manifold without boundary, it is clear that $\int_M\mathbf{V}_\Phi d\textrm{vol}_M=0$. The property of elliptical operator $$\Delta_\rho=\nabla\cdot(\rho\nabla)$$ shows that $\mathbf{V}_\Phi\colon \mathcal{F}(M)/\mathbb{R}\rightarrow T_\rho\mathcal{P}_+(M)$ is a well defined map,
linear, and one to one. In other words, $\mathcal{F}(M)/\mathbb{R}\cong T_\rho^*\mathcal{P}_+(M)$, where $T_\rho^*\mathcal{P}_+(M)$ is the smooth cotangent space of $\mathcal{P}_+(M)$. 

The identification induces the following inner product on 
$T_\rho\mathcal{P}_+(M)$. We first present this metric tensor in a \emph{dual} formulation \cite{Lott}. 
\begin{definition}[Inner product in dual coordinates]\label{d9}
The inner product 
$g_W :T_\rho\mathcal{P}_+(M)\times T_\rho\mathcal{P}_+(M)  \rightarrow \mathbb{R}$ 
takes any two tangent vectors $\sigma_1=\mathbf{V}_{\Phi_1}$ and $\sigma_2=\mathbf{V}_{\Phi_2}\in T_\rho\mathcal{P}_+(M)$ to 
\begin{equation*}\begin{split}\label{formula}
g_W(\sigma_1, \sigma_2)=\int_M\sigma_1\Phi_2d\textrm{vol}_M=\int_M\sigma_2^{}\Phi_1d\textrm{vol}_M=\int_M(\nabla\Phi_1, \nabla\Phi_2)\rho d\textrm{vol}_M.
\end{split} 
\end{equation*}
\end{definition}
Define $(-\Delta_\rho)^{\dd}\colon T_\rho\mathcal{P}_+(M)\rightarrow  T_\rho\mathcal{P}_+(M)$ the pseudo inverse operator of $(-\Delta_\rho)$. One simply check the fact that
\begin{equation*}
(-\Delta_\rho)^{\dd}(-\Delta_\rho)(-\Delta_\rho)^{\dd}=(-\Delta_{\rho})^{\dd}.
\end{equation*}
Thus 
\begin{equation*}
\begin{split}
\int_{M}(\nabla\Phi_1, \nabla\Phi_2)\rho d\textrm{vol}_M=&\int_M \Phi_1(-\Delta_\rho)\Phi_2d\textrm{vol}_M\\
=&\int_M \mathbf{V}_{\Phi_1} (-\Delta_\rho)^{\dd}(-\Delta_\rho)(-\Delta_\rho)^{\dd}\mathbf{V}_{\Phi_2} d\textrm{vol}_M\\
=&\int_M \sigma_1 (-\Delta_\rho)^{\dd}\sigma_2d\textrm{vol}_M.
\end{split}
\end{equation*}

Based on above understandings, we next present the metric tensor in {\em primal} coordinates. 
\begin{definition}[Inner product in primal coordinates]
Given $\sigma_1,\sigma_2\in T_\rho\mathcal{P}_+(M)$, the inner product $g_{W}(\cdot, \cdot):T_\rho\mathcal{P}_+(M)\times T_\rho\mathcal{P}_+(M)  \rightarrow \mathbb{R}$ is defined by
\begin{equation*}
g_{W}(\sigma_1,\sigma_2)=\int_{M}{\sigma_1} (-\Delta_\rho)^{\dagger}\sigma_2d\textrm{vol}_M.
\end{equation*}
\end{definition}
Following~\cite{Lafferty}, $(\mathcal{P}_+(M), g_W)$ is named density manifold. The variational problem from inner product gives a minimization of geometry energy functional in $\mathcal{P}_+(M)$.
\begin{equation*}
\begin{split}
E(\rho_t)=&\inf_{\rho_t\in \mathcal{P}_+(M)}\Big\{\int_0^1 \int_M \partial_t\rho_t(-\Delta_{\rho_t})^{\dd}\partial_t\rho_t d\textrm{vol}_Mdt~\colon~ \rho_0= \rho^0,~ \rho_1= \rho^1\Big\}\\
=&\inf_{\rho_t\in \mathcal{P}_+(M)}\Big\{\int_0^1\int_M(\nabla\Phi_t, \nabla\Phi_t)\rho_td\textrm{vol}_Mdt~\colon~\partial_t\rho_t+\nabla\cdot(\rho_t\nabla\Phi_t)=0,~ \rho_0= \rho^0,~ \rho_1= \rho^1\Big\}.
\end{split}
\end{equation*}
 The energy function equals the squared of geodesic distance, known as $L^2$-Wasserstein distance. In this case, the inverse Laplacian operator $(-\Delta_\rho)^{\dd}$ introduces the Legendre transformation in density manifold$$\Phi_t=(-\Delta_{\rho_t})^{\dd}\partial_t\rho_t.$$ 
As in previous section, $(\rho_t, \partial_t\rho_t)$ represents the primal coordinates in tangent bundle while $(\rho_t, \Phi_t)$ refers the dual coordinates in cotangent bundle. 

We note that the $L^2$-Wasserstein metric has many other equivalent formulations, including optimal mapping formulation, named Monge problem, and the statical formulation, called Kantorovich problem. For more details see \cite{vil2003}. In this paper, we focus on its induced metric tensor in primal coordinates. 
\subsection{Wasserstein Hamiltonian flows}
We next present the Hamiltonian flows in density manifold. We shall introduce the following second order partial differential equation \begin{equation}\label{SS}
\partial_{tt}\rho_t+\Gamma_W(\partial_t\rho_t, \partial_t\rho_t)=-\textrm{grad}_W\mathcal{F}(\rho_t),
\end{equation}
where $\Gamma_W$ is the Christopher symbol, representing the quadratic function of $\partial_t\rho_t$,  and $\textrm{grad}_W$ is the Riemannian gradient operator in $(\mathcal{P}_+(M), g_W)$. The above equation has been derived by a geometric approach in \cite{li_geometry}. In this paper, we would like to proceed with the other derivation based on Hamilton's variational principle.

Let $\mathcal{F}\colon \mathcal{P}_+(M)\rightarrow \mathbb{R}$ be a smooth potential energy. The Lagrangian in density manifold is given by
\begin{equation*}
\mathcal{L}(\rho_t, \partial_t\rho_t)=\frac{1}{2}g_W(\partial_t\rho_t, \partial_t\rho_t)-\mathcal{F}(\rho_t).
\end{equation*}
In above formula, $\mathcal{L}$ represents the kinetic energy minus potential energy in density manifold. It can be viewed as the ``expectation'' of Lagrangian in $M$ based on current probability density. Here the path $x(t)$ in $M$ is represented by the corresponding density path $\rho_t$.  

Consider the variational problem
\begin{equation}\label{Lag}
I(\rho_t)=\inf_{\rho_t}\Big\{\int_0^T\mathcal{L}(\rho_t, \partial_t\rho_t)dt\colon \rho_0=\rho^0,~\rho_T=\rho^T\Big\}.
\end{equation}
A density path is critical for $\mathcal{L}$ in case $I(\rho_t)$ is stationary for variations. We next derive Hamiltonian flows by finding critical paths of \eqref{Lag}.  
\begin{theorem}[Hamiltonian flow in primal coordinates]\label{ELD}
The Euler-Lagrange equation of variational problem \eqref{Lag} satisfies
\begin{equation}\label{EL}
{\partial_t}\frac{\delta}{\delta\partial_t\rho_t}\mathcal{L}(\rho_t, \partial_t\rho_t)=\frac{\delta}{\delta\rho_t}\mathcal{L}(\rho_t, \partial_t\rho_t)+C(t),
\end{equation}
where $\frac{\delta}{\delta\rho_t}$, $\frac{\delta}{\delta\partial_t\rho_t}$ is the $L^2$ first variation w.r.t. $\rho_t$, $\partial_t\rho_t$ respectively, and $C(t)$ is a spatially-constant function. More explicitly, the Euler-Lagrange equation can be rewritten as
\begin{equation}\label{geo}
\partial_{tt}\rho_t-\Big(\Delta_{\partial_t\rho_t}\Delta_{\rho_t}^{\dagger}\partial_t\rho_t+\frac{1}{2}\Delta_{\rho_t}(\nabla \Delta_{\rho_t}^{\dagger}\partial_t\rho_t)^2\Big)=\nabla\cdot(\rho_t\nabla\frac{\delta}{\delta\rho_t}\mathcal{F}(\rho_t)).
\end{equation}
\end{theorem}
\begin{remark}
By comparing \eqref{geo} with \eqref{SS}, we note that 
 $$\Gamma_W(\partial_t\rho_t, \partial_t\rho_t)=-\Big(\Delta_{\partial_t\rho_t}\Delta_{\rho_t}^{\dagger}\partial_t\rho_t+\frac{1}{2}\Delta_{\rho_t}(\nabla \Delta_{\rho_t}^{\dagger}\partial_t\rho_t)^2\Big),$$ while $$\textrm{grad}_W\mathcal{F}(\rho_t)=-\nabla\cdot(\rho_t\nabla\frac{\delta}{\delta\rho_t}\mathcal{F}(\rho_t)).$$
 \end{remark}
\begin{proof}
Denote a smooth perturbation function $h_t=h(t,\cdot)$, such that $\int_{M}h_td\textrm{vol}_M=0$ for all $t\in [0, T]$ and $h(0,\cdot)=h(T,\cdot)=0$. Denote $\rho^\epsilon_t=\rho^\epsilon(t,\cdot)=\rho_t+\epsilon h_t$, and consider the Taylor expansion of $I(\rho^\epsilon_t)$ w.r.t. $\epsilon$, 
\begin{equation*}
I(\rho^\epsilon_t)=I(\rho_t)+\epsilon\frac{d}{d\epsilon}I(\rho^\epsilon_t)|_{\epsilon=0} +o(\epsilon).
\end{equation*}
Notice that 
\begin{equation*}
\begin{split}
I(\rho^\epsilon_t)=&\int_0^T\mathcal{L}(\rho_t+\epsilon h_t, \partial_t\rho_t+\epsilon \partial_th_t)dt\\
=&\int_0^T\mathcal{L}(\rho_t, \partial_t\rho_t)dt+\epsilon\int_0^T\int_M\big(\frac{\delta}{\delta\rho_t}\mathcal{L}(\rho_t,\partial_t\rho_t)\cdot h_t+\frac{\delta}{\delta\partial_t\rho_t}\mathcal{L}(\rho_t,\partial_t\rho_t)\cdot\partial_t h_t\big) d\textrm{vol}_Mdt+o(\epsilon).
\end{split}
\end{equation*}
It is clear that $\frac{d}{d\epsilon}I(\rho^\epsilon_t)|_{\epsilon=0}=0$ implies 
\begin{equation*}
\int_0^T\int_M\big(\frac{\delta}{\delta\rho_t}\mathcal{L}(\rho_t,\partial_t\rho_t)\cdot h_t+\frac{\delta}{\delta\partial_t\rho_t}\mathcal{L}(\rho_t,\partial\rho_t)\cdot\partial_t h_t\big) d\textrm{vol}_Mdt=0.
\end{equation*}
Perform integration by parts w.r.t. $t$ in above formula and notice $h(0,x)=h(T, x)=0$. Then
\begin{equation*}
\int_0^T\int_M\big(\frac{\delta}{\delta\rho_t}\mathcal{L}(\rho_t,\partial_t\rho_t)-\partial_t\frac{\delta}{\delta\partial_t\rho_t}\mathcal{L}(\rho_t,\partial\rho_t)\big)h_t d\textrm{vol}_Mdt=0.
\end{equation*}
Since $\int_Mh_td\textrm{vol}_M=0$, then the equation \eqref{EL} holds up to a spatially-constant function shrift. 

We next derive \eqref{geo} by expressing \eqref{EL} explicitly. In other words, notice $\frac{\delta}{\delta\partial_t\rho_t}\mathcal{L}=(-\Delta_{\rho_t})^{\dd}\partial_t\rho_t$, then \eqref{EL} forms
\begin{equation*}\label{EL1}
\partial_t \big((-\Delta_{\rho_t})^{\dd}\partial_t\rho_t\big)=\frac{\delta}{\delta\rho_t}\big(\frac{1}{2}\int_M\partial_t\rho_t(-\Delta_{\rho_t})^{\dd}\partial_t\rho_t d\textrm{vol}_M -\mathcal{F}(\rho_t)\big)+C(t).
\end{equation*}
Towards the above equation, we shall show that its L.H.S. satisfies
\begin{equation}\label{aa}
\partial_t \big((-\Delta_{\rho_t})^{\dd}\partial_t\rho_t\big)=(-\Delta_{\rho_t})^{\dd}\partial_{tt}\rho_t-(-\Delta_{\rho_t})^{\dd}(-\Delta_{\partial_t\rho_t})(-\Delta_{\rho_t})^{\dd}\partial_t\rho_t,
\end{equation}
while the R.H.S. satisfies 
\begin{equation}\label{bb}
\frac{\delta}{\delta\rho_t}\big(\frac{1}{2}\int_M\partial_t\rho_t(-\Delta_{\rho_t})^{\dd}\partial_t\rho_t d\textrm{vol}_M -\mathcal{F}(\rho_t)\big)=-\frac{1}{2}(\nabla\Delta_{\rho_t}^{\dd}\partial_t\rho_t)^2-\frac{\delta}{\delta\rho_t}\mathcal{F}(\rho_t).
\end{equation}
Combining \eqref{aa} and \eqref{bb}, multiplying $\Delta_{\rho_t}$ on both sides and collecting all quadratic term of $\partial_t\rho_t$, we prove the result.

We next prove \eqref{aa} and \eqref{bb} by the following claim. 

\noindent\textbf{Claim:}
For any $\sigma\in T_\rho\mathcal{P}_+(M)$, then
\begin{equation*}
\Big(\partial_t(-\Delta_{\rho_t})^{\dd}\Big)\sigma=-(-\Delta_{\rho_t})^{\dd}(-\Delta_{\partial_t\rho_t})(-\Delta_{\rho_t})^{\dd}\sigma.
\end{equation*}
\begin{proof}[Proof of Claim]
Given $t\in[0, T]$, denote $\rho=\rho_t\in\mathcal{P}_+(M)$. Since $(-\Delta_{\rho})^{\dd}$ is semi-positive, we construct a positive self-adjoint operator $g(\rho)\colon C^\infty(M)\rightarrow C^\infty(M)$ to compute its derivative. Define 
 \begin{equation*}
g(\rho)f=\big(-\Delta_\rho\big)^{\dd}(f-\int_{M}fd\textrm{vol}_M)+\int_{M}fd\textrm{vol}_M, \quad \textrm{for $f\in C^{\infty}(M)$}.
\end{equation*}
We shall simply check that the inverse operator of $g(\rho)$ satisfies 
\begin{equation*}
g(\rho)^{-1}f=(-\Delta_\rho)f+\int_Mfd\textrm{vol}_M.
\end{equation*}
Notice
 \begin{equation*}
 \begin{split}
g(\rho)^{-1}g(\rho)f=&(-\Delta_\rho)\Big(\big(-\Delta_\rho\big)^{\dd}(f-\int_{M}fd\textrm{vol}_M)+\int_{M}fd\textrm{vol}_M\Big)+\int_{M}fd\textrm{vol}_M\\
=&f-\int_Mfd\textrm{vol}_M+(-\Delta_\rho\int_Mfd\textrm{vol}_M)+\int_Mfd\textrm{vol}_M\\
=&f-\int_Mfd\textrm{vol}_M+\int_Mfd\textrm{vol}_M=f.
\end{split}
\end{equation*}
Since $g(\rho)$ is a linear operator, then
 \begin{equation*}
\begin{split}
0=&\partial_tf=\partial_t(g(\rho_t)^{-1}g(\rho_t)f)\\
=&\partial_tg(\rho_t)^{-1} g(\rho_t)f+g(\rho_t)^{-1}\partial_tg(\rho_t)f.
\end{split}
\end{equation*}
Thus $\partial_tg(\rho_t)f=-g(\rho_t)\partial_tg(\rho_t)^{-1}g(\rho_t)f$. If $f=\sigma\in T_\rho\mathcal{P}_+(M)$, i.e. $\int_{M}\sigma d\textrm{vol}_M=0$, then $g(\rho)\sigma=(-\Delta_\rho)^{\dd}\sigma$. 
Thus 
\begin{equation*}\label{a}
\begin{split}
\partial_tg(\rho_t)\sigma=&-g(\rho_t)\partial_tg(\rho_t)^{-1}g(\rho_t)\sigma\\
=&-g(\rho_t)\partial_t(-\Delta_{\rho_t})g(\rho_t)\sigma\\ 
=&-(-\Delta_{\rho_t})^{\dd}(-\Delta_{\partial_t\rho_t})(-\Delta_{\rho_t})^{\dd}\sigma,
\end{split}
\end{equation*}
where the last equality is true since $\Delta_{\rho_t}=\nabla\cdot(\rho_t\nabla)$ is linear w.r.t. $\rho_t$. 
\end{proof}
We demonstrate \eqref{aa}. From the claim,
\begin{equation*}
\begin{split}
\partial_t \big((-\Delta_{\rho_t})^{\dd}\partial_t\rho_t\big)=&(-\Delta_{\rho_t})^{\dd}\partial_{t}(\partial_t\rho_t)+  \Big(\partial_t(-\Delta_{\rho_t})^{\dd}\Big)\partial_t\rho_t \\
=&(-\Delta_{\rho_t})^{\dd}\partial_{tt}\rho_t-(-\Delta_{\rho_t})^{\dd}(-\Delta_{\partial_t\rho_t})(-\Delta_{\rho_t})^{\dd}\partial_t\rho_t.
\end{split}
\end{equation*}
We show \eqref{bb}. Consider a perturbation function $h\in C^{\infty}(M)$, then  
\begin{equation*}
\begin{split}
&\frac{d}{d\epsilon}\int_M\partial_t\rho_t(-\Delta_{(\rho_t+\epsilon h)})^{\dd}\partial_t\rho_t d\textrm{vol}_M|_{\epsilon=0}\\
=&\int_M\partial_t\rho_t\big(-(-\Delta_{\rho_t+\epsilon h}^{\dd})(-\Delta_{h}) (-\Delta_{\rho_t+\epsilon h})^{\dd}\big) \partial_t\rho_t d\textrm{vol}_M|_{\epsilon=0}\\
=&\int_M\partial_t\rho_t\big(-(-\Delta_{\rho_t}^{\dd})(-\Delta_{h}) (-\Delta_{\rho_t})^{\dd}\big) \partial_t\rho_t d\textrm{vol}_M\\
=&\int_M(\Delta_{\rho_t}^{\dd}\partial_t\rho_t) (\Delta_h)(\Delta_{\rho_t}^{\dd}\partial_t\rho_t)d\textrm{vol}_M=\int_M (\Delta_{\rho_t}^{\dd}\partial_t\rho_t) \nabla\cdot (h\nabla\Delta_{\rho_t}^{\dd}\partial_t\rho_t)d\textrm{vol}_M\\
=&-\int_M (\nabla\Delta_{\rho_t}^{\dd}\partial_t\rho_t)^2h d\textrm{vol}_M,
\end{split}
\end{equation*}
where the first equality is shown by the claim. 
From the definition of $L^2$ first variation, \eqref{bb} is proved. 
\end{proof}

Secondly, we demonstrate that Euler-Lagrange \eqref{geo} can be recast into Hamilton's equations. 
\begin{proposition}[Hamiltonian flow in dual coordinates]\label{leg}
Consider 
\begin{equation*}
\Phi_t=\big(-\Delta_{\rho_t}\big)^{\dd}\partial_t\rho_t,
\end{equation*}
then equation \eqref{geo} can be formulated as the first order system of $(\rho_t, \Phi_t)$, 
\begin{equation*}
\begin{cases}
&\partial_t\rho_t+\nabla\cdot(\rho_t\nabla\Phi_t)=0\\
&\partial_t\Phi_t+\frac{1}{2}(\nabla\Phi_t)^2=-\frac{\delta}{\delta\rho_t}\mathcal{F}(\rho_t)
\end{cases},
\end{equation*}
where $\Phi_t$ is up to a spatially-constant function shrift. In other words, 
\begin{equation*}
\partial_t\rho_t=\frac{\delta}{\delta\Phi_t}\mathcal{H}(\rho_t, \Phi_t),\quad \partial_t\Phi_t=-\frac{\delta}{\delta\rho_t}\mathcal{H}(\rho_t, \Phi_t),
\end{equation*}
where the Hamiltonian is given by 
\begin{equation*}
\mathcal{H}(\rho_t, \Phi_t)=\int_{M}\frac{1}{2}(\nabla\Phi_t)^2\rho_td\textrm{vol}_M+\mathcal{F}(\rho_t).
\end{equation*}
\end{proposition}
\begin{proof}
We directly check the result. Since $\Phi_t=(-\Delta_{\rho_t})^{\dd}\partial_t\rho_t$, then 
the continuity equation holds, i.e. $\partial_t\rho_t+\Delta_{\rho_t}\Phi_t=\partial_t\rho_t+\nabla\cdot(\rho_t\nabla\Phi_t)=0$. We only need to show that $\Phi_t$ up to a spatially-constant function shrift satisfies the Hamilton-Jacobi equation. We rewrite \eqref{geo} by
\begin{equation*}
\begin{split}
0=&\partial_t(\partial_t\rho_t)-\Delta_{\partial_t\rho_t}\Delta_{\rho_t}^{\dagger}\partial_t\rho_t-\frac{1}{2}\Delta_{\rho_t}(\nabla \Delta_{\rho_t}^{\dagger}\partial_t\rho_t)^2-\Delta_{\rho_t}\frac{\delta}{\delta_{\rho_t}}\mathcal{F}(\rho_t)\\
=&-\partial_t(\Delta_{\rho_t}\Phi_t)+\Delta_{\partial_t\rho_t}\Phi_t-\frac{1}{2}\Delta_{\rho_t}(\nabla\Phi_t)^2-\Delta_{\rho_t}\frac{\delta}{\delta\rho_t}\mathcal{F}(\rho_t)\\
=&-\Delta_{\rho_t}\Big(\partial_t\Phi_t+\frac{1}{2}(\nabla\Phi_t)^2+\frac{\delta}{\delta\rho_t}\mathcal{F}(\rho_t)\Big).
\end{split}
\end{equation*}
Based on the property of elliptical operator $(-\Delta_{\rho_t})$, $\Phi_t$ up to a spatially-constant function shrift satisfies the Hamilton-Jacobi equation. 
\end{proof}

In the last, we demonstrate a natural mathematical connection between Hamiltonian flows in $(\mathcal{P}_+(M), g_{W})$ and $(M, g)$. We shall show that the density transition equation of a second order ODE \eqref{HF} satisfies a second order PDE \eqref{geo}. For illustration, let $(M, g)$ be a $d$ dimensional torus $(\mathbb{T}^d, \mathbb{I})$.
\begin{proposition}[Hamiltonian flow as density transition equation]\label{pt}
Let $(X_t)_{0\leq t<T}$ be a smooth diffeomorphism in $\mathbb{T}^d$ with $X_0=\textrm{Id}$, $\dot X_0=\nabla\Phi$ for some smooth function $\Phi(x)$. Suppose $X_t$ satisfies 
\begin{equation}\label{gd}
\frac{d^2}{dt^2} X_t=-\nabla_{X_t}\frac{\delta}{\delta\rho(t,X_t)}\mathcal{F}(\rho_t).
\end{equation}
Given the initial density $\mu\in\mathcal{P}_+(\mathbb{T}^d)$, $\rho_t=X_t\#\mu$, i.e. $\rho_t$ equals $X_t$ push-forward $\mu$. Then the density path $\rho_t=\rho(t,\cdot)$ is a solution of \eqref{geo}. 
\end{proposition}
\begin{proof}
Denote $\frac{d}{dt}X_t(x)=v(t,X_t(x))$, \eqref{gd} can be rewritten as  
\begin{equation*}
\begin{cases}
&\frac{d}{dt}X_t(x)=v(t, X_t(x))\\
&\frac{d}{dt}v(t, X_t(x))=-\nabla_{X_t}\frac{\delta}{\delta\rho(t,X_t(x))}\mathcal{F}(\rho_t).
\end{cases}
\end{equation*}
On one hand, by differentiating $\frac{d}{dt}X_t(x)=v(t, X_t(x))$, we obtain for any $x$, 
\begin{equation*}
\begin{split}
-\nabla_{X_t}\frac{\delta}{\delta\rho(t,X_t(x))}\mathcal{F}(\rho_t)=&\frac{d}{dt}v(t,X_t(x))\\
=&\partial_tv(t, X_t(x))+\nabla v(t,X_t(x))\cdot \frac{d}{dt}X_t(x)\\
=&\partial_tv(t, X_t(x))+\nabla v(t,X_t(x))\cdot v(t,X_t(x)).
\end{split}
\end{equation*}
Thus 
\begin{equation}\label{v}
\partial_tv(t,x)+\nabla v(t,x)\cdot v(t,x)=-\nabla\frac{\delta}{\delta\rho(t,x)}\mathcal{F}(\rho).
\end{equation}
On the other hand, we demonstrate that $\rho_t=X_t\#\mu$ solves the continuity equation 
\begin{equation}\label{c}
\partial_t\rho(t,x)+\nabla\cdot(\rho(t,x) v(t,x))=0.
\end{equation}
We shall show that for any test function $\psi\in C^{\infty}(\mathbb{T}^d)$, 
\begin{equation*}
\frac{d}{dt}\int_{\mathbb{T}^d}\psi(x)\rho(t,x)dx=-\int_{\mathbb{T}^d}\nabla\cdot(\rho(t,x)v(t,x))\psi(t,x)dx\equiv\int_{\mathbb{T}^d}(\nabla\psi(t,x)\cdot v(t,x))\rho(t,x)dx.
\end{equation*}
By the definition of push-forward $\rho_t=X_t\#\mu$, we have 
\begin{equation*}
\int_{\mathbb{T}^d}\psi(x)\rho(t,x)dx=\int_{\mathbb{T}^d}\psi(X_t(x))\mu(x)dx.
\end{equation*}
Then
\begin{equation*}
\begin{split}
\frac{d}{dt}\int_{\mathbb{T}^d}\psi(x)\rho(t,x)dx=&\frac{d}{dt}\int_{\mathbb{T}^d}\psi(X_t(x))\mu(x)dx=\int_{\mathbb{T}^d}\frac{d}{dt}\psi(X_t(x))\mu(x)dx\\
=&\int_{\mathbb{T}^d}\big(\nabla\psi(X_t(x))\cdot \frac{d}{dt}X_t(x)\big)\mu(x)dx\\
=&\int_{\mathbb{T}^d}\big(\nabla\psi(X_t(x))\cdot v(t,X_t(x))\big)\mu(x)dx\\
=&\int_{\mathbb{T}^d}\big(\nabla\psi(x)\cdot v(t,x)\big)\rho(t,x)dx,
\end{split}
\end{equation*}
where the last equality is from the definition of push-forward, i.e. $\rho_t=X_t\#\mu$. 
Thus $(\rho(t,x), v(t,x))$ solves the system of \eqref{v} and \eqref{c}.

We next demonstrate that the system of \eqref{v}, \eqref{c} can be written into a single equation \eqref{geo}. We first construct a function $\Phi(t,x)$ such that $v(t,x)=\nabla\Phi(t,x)$, $\Phi(0,x)=\Phi(x)$. We check that equation \eqref{v} is equivalent to 
\begin{equation*}
\nabla\big(\partial_t\Phi_t+\frac{1}{2}(\nabla\Phi_t)^2+\frac{\delta}{\delta{\rho_t}}\mathcal{F}(\rho_t)\big)=0.
\end{equation*}
In other words, $$\partial_t\Phi_t+\frac{1}{2}(\nabla\Phi_t)^2+\frac{\delta}{\delta{\rho_t}}\mathcal{F}(\rho_t)=C(t),$$ for some constant function $C(t)$.
From \eqref{c}, $\Phi_t=(-\Delta_{\rho_t})^{\dd}\partial_t\rho_t$. Substituting it into the above, we have  
\begin{equation*}
\partial_t\big((-\Delta_{\rho_t})^{\dd}\partial_t\rho_t\big)+\frac{1}{2}\Big(\nabla ((-\Delta_{\rho_t})^{\dd}\partial_t\rho_t)^2  \Big)+\frac{\delta}{\delta{\rho_t}}\mathcal{F}(\rho_t)=C(t).
\end{equation*}
From the equality \eqref{aa}, we derive 
\begin{equation*}
(-\Delta_{\rho_t})^{\dd}\partial_{tt}\rho_t+\Delta_{\rho_t}^{\dd}\Delta_{\partial_t\rho_t}\Delta_{\rho_t}^{\dd}\partial_t\rho_t+\frac{1}{2}(\nabla\Delta_{\rho_t}^{\dd}\partial_t\rho_t)^2+\frac{\delta}{\delta{\rho_t}}\mathcal{F}(\rho_t)=C(t).
\end{equation*}
Applying operator $(-\Delta_{\rho_t})$ on both sides of the above equation, we prove that $\rho_t$ satisfies equation \eqref{geo}.
\end{proof}
\section{Examples}\label{sec4}
In this section, we demonstrate that many well-known equations related to densities can be recast in the formalism of Hamiltonian flows in density manifold.
\begin{example}[Linear Vlasov equation]
Given a potential $V\in C^{\infty}(\mathbb{T}^d)$. Consider a linear Vlasov equation 
\begin{equation*}
\frac{\partial f(t,x,v)}{\partial t}+v\cdot \nabla_x f(t,x,v)-\nabla V(x)\cdot \nabla_vf(t,x,v)=0.
\end{equation*}
It represents the evolutionary of density $f(t,x,v)$ on $\mathbb{T}_x^d\times \mathbb{R}_v^d$ for particles moving with a force based on a potential. In other words, $f(t,x,v)$ is the transition density of  $(X_t, v_t)$ satisfying
\begin{equation*}
\begin{cases}
&\frac{d}{dt} X_t=v_t\\
&\frac{d}{dt} v_t=-\nabla V(X_t).
\end{cases}
\end{equation*}
On the other hand, the first order ODE system can be rewritten as the second order ODE
\begin{equation*}
\ddot X_t=-\nabla V(X_t).
\end{equation*}
From Proposition \ref{pt}, the density of  $X_t$ on $\mathbb{T}_x^d$, i.e. $\rho(t,x)=\int_{\mathbb{R}^d}f(t,x,v)dv$, satisfies the transition equation
\begin{equation*}
\partial_{tt}\rho_t-\Big(\Delta_{\partial_t\rho_t}\Delta_{\rho_t}^{\dagger}\partial_t\rho_t+\frac{1}{2}\Delta_{\rho_t}(\nabla \Delta_{\rho_t}^{\dagger}\partial_t\rho_t)^2\Big)=\nabla\cdot(\rho_t\nabla V(x)).
\end{equation*}
It is a Hamiltonian flow \eqref{SS} in density manifold w.r.t. the linear potential energy 
\begin{equation*}
\mathcal{F}(\rho)=\int_{\mathbb{T}^d}V(x)\rho(x)dx.
\end{equation*}
\end{example}
\begin{example}[Nonlinear Vlasov equation]
Given an interaction potential $W\in C^{\infty}(\mathbb{T}^d)$. Consider a nonlinear Vlasov equation 
\begin{equation*}
\begin{cases}
&\frac{\partial f(t,x,v)}{\partial t}+v\cdot \nabla_x f(t,x,v)-\nabla \bar W(x,\rho)\cdot \nabla_vf(t,x,v)=0\\
&\bar W(x,\rho)=\int_{\mathbb{T}^d}W(|x-y|)\rho(t,y)dy, ~\rho(t,x)=\int_{\mathbb{R}^d}f(t,x,v)dv.
\end{cases}
\end{equation*}
The above equation represents that particles evolve with a force based on an interaction potential $\bar W$, which is created by all of particles. In this case, $f(t,x,v)$ is the density equation of $(X_t, v_t)$ satisfying
\begin{equation*}
\begin{cases}
&\frac{d}{dt} X_t=v_t\\
&\frac{d}{dt} v_t=-\nabla \bar W(X_t, \rho_t),
\end{cases}
\end{equation*}
where $\rho_t$ is the density function of $X_t$. Similar as the first example, the density of  $X_t$ satisfies the transition equation 
\begin{equation*}
\partial_{tt}\rho_t-\Big(\Delta_{\partial_t\rho_t}\Delta_{\rho_t}^{\dagger}\partial_t\rho_t+\frac{1}{2}\Delta_{\rho_t}(\nabla \Delta_{\rho_t}^{\dagger}\partial_t\rho_t)^2\Big)=\nabla\cdot\Big(\rho_t\nabla \bar W(x,\rho_t)\Big).
\end{equation*}
It is a Hamiltonian flow in density manifold w.r.t. the interaction potential energy 
\begin{equation*}
\mathcal{F}(\rho)=\frac{1}{2}\int_{\mathbb{T}^d}\int_{\mathbb{T}^d}W(|x-y|)\rho(x)\rho(y)dxdy.
\end{equation*}
\end{example}

\begin{example}[Schr{\"o}dinger equation]
Given a potential $V\in C^{\infty}(\mathbb{T}^d)$. Consider a linear Schr{\"o}dinger equation 
\begin{equation*}
i\partial_t\Psi(t,x)=-\frac{1}{2}\Delta\Psi(t,x)+V(x)\Psi(t,x).
\end{equation*}
 Here $\Psi$ is the complex wave function of the quantum system. The complex wave equation can be related to the density function by ``Madelung'' (Bohm) transform \begin{equation*}
\Psi(t,x)=\sqrt{\rho(t,x)}e^{-i\Phi(t,x)}.
\end{equation*}
Here $\rho(t,x)$ is a density function on $\mathbb{T}_x^d$ and $\Phi(t,x)$ is a potential function. Then $(\rho(t,x), \Phi(t,x))$ satisfies the following pair of equations
\begin{equation*}
\begin{cases}
&\partial_t\rho_t+\nabla\cdot(\rho_t\nabla\Phi_t)=0\\
&\partial_t\Phi_t+\frac{1}{2}(\nabla\Phi_t)^2=-V(x)-\frac{1}{8}\frac{\delta}{\delta\rho_t}\int_{\mathbb{T}^d}(\nabla\log\rho_t)^2\rho_tdx.
\end{cases}
\end{equation*}
Here $\int_{\mathbb{T}^d}(\nabla\log\rho)^2\rho dx$ represents a functional in density manifold, named Fisher information.  
From Proposition \ref{leg}, the density $\rho(t,x)$ satisfies 
\begin{equation*}
\partial_{tt}\rho_t-\Big(\Delta_{\partial_t\rho_t}\Delta_{\rho_t}^{\dagger}\partial_t\rho_t+\frac{1}{2}\Delta_{\rho_t}(\nabla \Delta_{\rho_t}^{\dagger}\partial_t\rho_t)^2\Big)=\nabla\cdot\Big(\rho_t\nabla (V(x)+\frac{1}{8}\frac{\delta}{\delta\rho_t}\int_{\mathbb{T}^d}(\nabla\log\rho_t)^2\rho_tdx)\Big).
\end{equation*}
It is a Hamiltonian flow \eqref{SS} in density manifold w.r.t. the linear potential energy plus the Fisher information
\begin{equation*}
\mathcal{F}(\rho)=\int_{\mathbb{T}^d}V(x)\rho(x)dx+\frac{1}{8}\int_{\mathbb{T}^d}(\nabla\log\rho(x))^2\rho(x)dx.
\end{equation*}
Similar formulation is also true for nonlinear Schr{\"o}dinger equations. 
\end{example}

\begin{example}[Schr{\"o}dinger Bridge problem]
Consider a Schr{\"o}dinger system \cite{CP}
\begin{equation*}
\partial_t\eta_t=\frac{1}{2}\Delta\eta_t,\qquad\partial_t\eta^*_t=-\frac{1}{2}\Delta\eta^*_t.
\end{equation*}
 Here $\eta, \eta^*\colon \mathbb{T}^d\rightarrow \mathbb{R}$ are real value functions. The complex wave equation can be related to the density function by ``Hopf-Cole'' transformation
 \begin{equation*}
\eta=\sqrt{\rho}e^{S/2},\quad \eta^*=\sqrt{\rho}e^{-S/2}.
\end{equation*}
Here $\rho(t,x)$ is a density function on $\mathbb{T}_x^d$ and $\Phi(t,x)$ is a potential function. Then $(\rho(t,x), \Phi(t,x))$ satisfies the following pair of equations
\begin{equation*}
\begin{cases}
&\partial_t\rho_t+\nabla\cdot(\rho_t\nabla\Phi_t)=0\\
&\partial_t\Phi_t+\frac{1}{2}(\nabla\Phi_t)^2=\frac{1}{8}\frac{\delta}{\delta\rho_t}\int_{\mathbb{T}^d}(\nabla\log\rho_t)^2\rho_tdx.
\end{cases}
\end{equation*}
Here $\int_{\mathbb{T}^d}(\nabla\log\rho)^2\rho dx$ represents a functional in density manifold, named Fisher information.  
From Proposition \ref{leg}, the density $\rho(t,x)$ satisfies 
\begin{equation*}
\partial_{tt}\rho_t-\Big(\Delta_{\partial_t\rho_t}\Delta_{\rho_t}^{\dagger}\partial_t\rho_t+\frac{1}{2}\Delta_{\rho_t}(\nabla \Delta_{\rho_t}^{\dagger}\partial_t\rho_t)^2\Big)=-\nabla\cdot\big(\rho_t\nabla(\frac{1}{8}\frac{\delta}{\delta\rho_t}\int_{\mathbb{T}^d}(\nabla\log\rho_t)^2\rho_tdx)\big).
\end{equation*}
It is a Hamiltonian flow \eqref{SS} in density manifold w.r.t. negative Fisher information
\begin{equation*}
\mathcal{F}(\rho)=-\frac{1}{8}\int_{\mathbb{T}^d}(\nabla\log\rho(x))^2\rho(x)dx.
\end{equation*}
\end{example}

\section{Discussions}
To summarize, we demonstrate the Euler-Lagrange equations, and associated Hamiltonian flows in density manifold with Lagrangian formalism. We show that the Hamiltonian flows in density space are probability transition equations of classical Hamiltonian ODEs. It mathematically demonstrates the intuition: 
The density of Hamiltonian flow in sample space is Hamiltonian flow in density manifold.

\textbf{Acknowledgments:} The authors thank Prof. Chongchun Zeng for many stimulating discussions. 


\begin{thebibliography}{10}

\bibitem{Carlen}
E.~A. Carlen.
\newblock Conservative {{Diffusions}}.
\newblock {\em Communications in Mathematical Physics}, 94(3):293--315, 1984.

\bibitem{li_SE}
S.-N. Chow, W.~Li, and H.~Zhou.
\newblock A discrete {{Schr\"odinger}} equation via optimal transport on
  graphs.
\newblock {\em Journal of Functional Analysis}, 276(8):2440--2469, 2019.

\bibitem{CP}
G.~Conforti and M.~Pavon.
\newblock Extremal flows on {{Wasserstein}} space.
\newblock {\em arXiv:1712.02257 [math-ph]}, 2017.

\bibitem{Lafferty}
J.~D. Lafferty.
\newblock The {{Density Manifold}} and {{Configuration Space Quantization}}.
\newblock {\em Transactions of the American Mathematical Society},
  305(2):699--741, 1988.

\bibitem{li_geometry}
W.~Li.
\newblock Geometry of probability simplex via optimal transport.
\newblock {\em arXiv:1803.06360 [math]}, 2018.

\bibitem{Lott}
J.~Lott.
\newblock Some {{Geometric Calculations}} on {{Wasserstein Space}}.
\newblock {\em Communications in Mathematical Physics}, 277(2):423--437, 2008.

\bibitem{Nelson0}
E.~Nelson.
\newblock Derivation of the {{Schr{\"o}dinger Equation}} from {{Newtonian
  Mechanics}}.
\newblock {\em Physical Review}, 150(4):1079--1085, 1966.

\bibitem{Nelson1}
E.~Nelson.
\newblock The free {{Markoff}} field.
\newblock {\em Journal of Functional Analysis}, 12(2):211--227, 1973.

\bibitem{Nelson}
E.~Nelson.
\newblock {\em Quantum Fluctuations}.
\newblock Princeton series in physics. {Princeton University Press}, Princeton,
  N.J, 1985.

\bibitem{Nelson3}
E.~Nelson.
\newblock Field theory and the future of stochastic mechanics.
\newblock In S.~Albeverio, G.~Casati, and D.~Merlini, editors, {\em Stochastic
  {{Processes}} in {{Classical}} and {{Quantum Systems}}}, volume 262, pages
  438--469. {Springer Berlin Heidelberg}, Berlin, Heidelberg, 1986.

\bibitem{Otto}
F.~Otto.
\newblock The {{Geometry}} of {{Dissipative Evolution Equations}}: {{The Porous
  Medium Equation}}.
\newblock {\em Communications in Partial Differential Equations},
  26(1-2):101--174, 2001.

\bibitem{vil2003}
C.~Villani.
\newblock {\em Topics in Optimal Transportation}.
\newblock Number v. 58 in Graduate studies in mathematics. {American
  Mathematical Society}, Providence, RI, 2003.

\bibitem{vil2008}
C.~Villani.
\newblock {\em Optimal Transport: Old and New}.
\newblock Number 338 in Grundlehren der mathematischen Wissenschaften.
  {Springer}, Berlin, 2009.

\end{thebibliography}
\end{document}